\documentclass{amsart}

\usepackage{amsmath,amsthm,amsfonts}
\usepackage{amssymb}
\usepackage[all]{xy}
\usepackage{verbatim}

\theoremstyle{plain}
\newtheorem{thm}{Theorem}
\newtheorem*{ithm}{Theorem}

\newtheorem{prop}[thm]{Proposition}

\theoremstyle{definition}

\newcommand{\acat}[1]{\mathsf{a}_{#1}}
\newcommand{\cacat}[1]{\mathsf{\hat{a}}_{#1}}
\newcommand{\lcat}{\mathsf{l}}
\newcommand{\sets}{\mathsf{Sets}}

\newcommand{\sfam}{\mathcal{M}}

\newcommand{\defs}{\mathsf{Def}_{\sfam}}
\newcommand{\defm}[1]{\mathsf{Def}_{#1}}

\DeclareMathOperator{\enm}{End}
\DeclareMathOperator{\ext}{Ext}
\DeclareMathOperator{\gr}{gr}
\DeclareMathOperator{\hmm}{Hom}
\DeclareMathOperator{\id}{id}

\DeclareMathOperator*{\osum}{\oplus}

\DeclareMathOperator{\rad}{rad}

\begin{document}

\title[The Generalized Burnside Theorem]{The Generalized Burnside
Theorem in noncommutative deformation theory}
\author{Eivind Eriksen}
\maketitle

\begin{abstract}
Let $A$ be an associative algebra over a field $k$, and let $\sfam$ be
a finite family of right $A$-modules. Study of the noncommutative
deformation functor $\defs: \acat r \to \sets$ of the family $\sfam$
leads to the construction of the algebra $\mathcal O^A(\sfam)$ of
observables and the Generalized Burnside Theorem, due to Laudal
\cite{laud02}. In this paper, we give an overview of aspects of
noncommutative deformations closely connected to the Generalized
Burnside Theorem.
\end{abstract}

\section{Introduction}

Let $k$ be a field and let $A$ be an associative $k$-algebra. For any
right $A$-module $M$, there is a commutative deformation functor $\defm
M: \lcat \to \sets$ defined on the category $\lcat$ of local Artinian
commutative $k$-algebras with residue field $k$. We recall that for an
algebra $R$ in $\lcat$, a deformation of $M$ to $R$ is a pair $(M_R,
\tau)$, where $M_R$ is an $R$-$A$ bimodule (on which $k$ acts
centrally) that is $R$-flat, and $\tau: k \otimes_R M_R \to M$ is an
isomorphism of right $A$-modules.

Let $\acat r$ be the category of $r$-pointed Artinian $k$-algebras for
$r \ge 1$, the natural noncommutative generalization of $\lcat$. We
recall that an algebra $R$ in $\acat r$ is an Artinian ring, together
with a pair of structural ring homomorphisms $f: k^r \to R$ and $g: R
\to k^r$ with $g \circ f = \id$, such that the radical $I(R) = \ker(g)$
is nilpotent. Any algebra $R$ in $\acat r$ has $r$ simple left modules
of dimension one, the natural projections $\{ k_1, \dots, k_r \}$ of
$k^r$.

In Laudal \cite{laud02}, a noncommutative deformation functor $\defs:
\acat r \to \sets$ of a finite family $\sfam = \{ M_1, \dots, M_r \}$
of right $A$-modules was introduced, as a generalization of the
commutative deformation functor $\defm M: \lcat \to \sets$ of a right
$A$-module $M$. In the case $r = 1$, this generalization is completely
natural, and can be defined word for word as in the commutative case.
The generalization to the case $r > 1$ is less obvious and has
further-reaching consequences, but is still very natural. A deformation
of $\sfam$ to $R$ is defined to be a pair $(M_R, \{ \tau_i \}_{1\le i
\le r})$, where $M_R$ is an $R$-$A$ bimodule (on which $k$ acts
centrally) that is $R$-flat, and $\tau_i: k_i \otimes_R M_R \to M_i$ is
an isomorphism of right $A$-modules for $1 \le i \le r$. We remark that
$M_R$ is $R$-flat if and only if
\begin{equation*}
M_R \cong (R_{ij} \otimes_k M_j) = \left(\begin{matrix} R_{11} \otimes_k
M_1 & R_{12} \otimes_k M_2 & \dots & R_{1r} \otimes_k M_r \\ R_{21}
\otimes_k M_1 & R_{22} \otimes_k M_2 & \dots & R_{2r} \otimes_k M_r \\
\vdots & \vdots & \ddots & \dots \\ R_{r1} \otimes_k M_1 & R_{r2}
\otimes_k M_2 & \dots & R_{rr} \otimes_k M_r \end{matrix}\right)
\end{equation*}
considered as a left $R$-module, and that a deformation in $\defs(R)$
may be thought of as a right multiplication $A \to \enm_R(M_R)$ of $A$
on the left $R$-module $M_R$ that lifts the multiplication $\rho: A \to
\osum_i \, \enm_k(M_i)$ of $A$ on the family $\sfam$.

There is an obstruction theory for $\defs$, generalizing the
obstruction theory for the commutative deformation functor. Hence there
exists a formal moduli $(H, M_H)$ for $\defs$ (assuming a mild
condition on $\sfam$). We consider the \emph{algebra of observables}
$\mathcal O^A(\sfam) = \enm_H(M_H) \cong (H_{ij} \widehat \otimes_k
\hmm_k(M_i,M_j) )$ and the commutative diagram
\begin{equation*}
\xymatrix{
A \ar[r]^-{\eta} \ar[rd]_-{\rho} & \mathcal O^A(\sfam) \ar[d]^-{\pi} \\
& \displaystyle \osum_{1 \le i \le r} \enm_k(M_i)
}
\end{equation*}
given by the versal family $M_H \in \defs(H)$. The algebra $B =
\mathcal O^A(\sfam)$ has an induced right action on the family $\sfam$
extending the action of $A$, and we may consider $\sfam$ as a family of
right $B$-modules. In fact, $\sfam$ is the family of simple $B$-modules
since $\pi$ can be identified with the quotient morphism $B \to B/\rad
B$.

When $A$ is an algebra of finite dimension over an algebraically closed
field $k$ and $\sfam$ is the family of simple right $A$-modules, Laudal
proved the \emph{Generalized Burnside Theorem} in Laudal \cite{laud02},
generalizing the structure theorem for semi-simple algebras and the
classical Burnside Theorem. Laudal's result be stated in the following
form:

\begin{ithm}[The Generalized Burnside Theorem]
Let $A$ be a finite-dimensional algebra over a field $k$, and let
$\sfam = \{ M_1, M_2, \dots, M_r \}$ be the family of simple right
$A$-modules. If $\enm_A(M_i) = k$ for $1 \le i \le r$, then $\eta: A
\to \mathcal O^A(\sfam)$ is an isomorphism. In particular, $\eta$ is an
isomorphism when $k$ is algebraically closed.
\end{ithm}

Let $A$ be an algebra of finite dimension over an algebraically closed
field $k$ and let $\sfam$ be any finite family of right $A$-modules of
finite dimension over $k$. Then the algebra $B = \mathcal O^A(\sfam)$
has the property that $\eta_B: B \to \mathcal O^B(\sfam)$ is an
isomorphism, or equivalently, that the assignment $(A, \sfam) \mapsto
(B, \sfam)$ is a closure operation. This means that the family $\sfam$
has exactly the same module-theoretic properties, in terms of (higher)
extensions and Massey products, considered as a family of modules over
$B$ as over $A$.

\section{Noncommutative deformations of modules}

Let $k$ be a field. For any integer $r \ge 1$, we consider the category
$\acat r$ of $r$-pointed Artinian $k$-algebras. We recall that an
object in $\acat r$ is an Artinian ring $R$, together with a pair of
structural ring homomorphisms $f: k^r \to R$ and $g: R \to k^r$ with $g
\circ f = \id$, such that the radical $I(R) = \ker(g)$ is nilpotent.
The morphisms of $\acat r$ are the ring homomorphisms that commute with
the structural morphisms. It follows from this definition that $I(R)$
is the Jacobson radical of $R$, and therefore that the simple left
$R$-modules are the projections $\{ k_1, \dots, k_r \}$ of $k^r$.

Let $A$ be an associative $k$-algebra. For any family $\sfam = \{ M_1,
\dots, M_r \}$ of right $A$-modules, there is a noncommutative
deformation functor $\defs: \acat r \to \sets$, introduced in Laudal
\cite{laud02}; see also Eriksen \cite{erik06}. For an algebra $R$ in
$\acat r$, we recall that a deformation of $\sfam$ over $R$ is a pair
$(M_R, \{ \tau_i \}_{1\le i \le r})$, where $M_R$ is an $R$-$A$
bimodule (on which $k$ acts centrally) that is $R$-flat, and $\tau_i:
k_i \otimes_R M_R \to M_i$ is an isomorphism of right $A$-modules for
$1 \le i \le r$. Moreover, $(M_R, \{ \tau_i \} ) \sim (M'_R, \{
\tau'_i\} )$ are equivalent deformations over $R$ if there is an
isomorphism $\eta: M_R \to M'_R$ of $R$-$A$ bimodules such that $\tau_i
= \tau'_i \circ (1 \otimes \eta)$ for $1 \le i \le r$. One may prove
that $M_R$ is $R$-flat if and only if
\begin{equation*}
M_R \cong (R_{ij} \otimes_k M_j) = \left(\begin{matrix} R_{11} \otimes_k
M_1 & R_{12} \otimes_k M_2 & \dots & R_{1r} \otimes_k M_r \\ R_{21}
\otimes_k M_1 & R_{22} \otimes_k M_2 & \dots & R_{2r} \otimes_k M_r \\
\vdots & \vdots & \ddots & \dots \\ R_{r1} \otimes_k M_1 & R_{r2}
\otimes_k M_2 & \dots & R_{rr} \otimes_k M_r \end{matrix}\right)
\end{equation*}
considered as a left $R$-module, and a deformation in $\defs(R)$ may be
thought of as a right multiplication $A \to \enm_R(M_R)$ of $A$ on the
left $R$-module $M_R$ that lifts the multiplication $\rho: A \to
\osum_i \, \enm_k(M_i)$ of $A$ on the family $\sfam$.

Let us assume that $\sfam$ is a \emph{swarm}, i.e. that
$\ext^1_A(M_i,M_j)$ has finite dimension over $k$ for $1 \le i,j \le
r$. Then $\defs$ has a pro-representing hull or a formal moduli $(H,
M_H)$, see Laudal \cite{laud02}, Theorem 3.1. This means that $H$ is a
complete $r$-pointed $k$-algebra in the pro-category $\cacat r$, and
that $M_H \in \defs(H)$ is a family defined over $H$ with the following
versal property: For any algebra $R$ in $\acat r$ and any deformation
$M_R \in
\defs(R)$, there is a homomorphism $\phi: H \to R$ such that
$\defs(\phi)(M_H) = M_R$. The formal moduli $(H, M_H)$ is unique up to
non-canonical isomorphism. However, the morphism $\phi$ is not uniquely
determined by $(R,M_R)$.

When $\sfam$ is a swarm with formal moduli $(H,M_H)$, right
multiplication on the $H$-$A$ bimodule $M_H$ by elements in $A$
determines an algebra homomorphism
    \[ \eta: A \to \enm_H(M_H) \]
We write $\mathcal O^A(\sfam) = \enm_H(M_H)$ and call it the
\emph{algebra of observables}. Since $M_H$ is $H$-flat, we have that
$\enm_H(M_H) \cong ( H_{ij} \widehat \otimes_k \hmm_k(M_i,M_j) )$, and
it follows that $\mathcal O^A(\sfam)$ is explicitly given as the matrix
algebra
\begin{equation*}
\left( \begin{matrix} H_{11} \widehat
\otimes_k \enm_k(M_1) & H_{12} \widehat \otimes_k \hmm_k(M_1,M_2) &
\dots & H_{1r} \widehat \otimes_k \hmm_k(M_1,M_r) \\ H_{21}
\widehat \otimes_k \hmm_k(M_2,M_1) & H_{22} \widehat \otimes_k
\enm_k(M_2) & \dots & H_{2r} \widehat \otimes_k \hmm_k(M_2,M_r) \\
\vdots & \vdots & \ddots & \dots \\ H_{r1} \widehat \otimes_k
\hmm_k(M_r,M_1) & H_{r2} \widehat \otimes_k \hmm_k(M_r,M_2) & \dots
& H_{rr} \widehat \otimes_k \enm_k(M_r) \end{matrix}\right)
\end{equation*}
Let us write $\rho_i: A \to \enm_k(M_i)$ for the structural algebra
homomorphism defining the right $A$-module structure on $M_i$ for $1
\le i \le r$, and
    \[ \rho: A \to \osum_{1 \le i \le r} \, \enm_k(M_i) \]
for their direct sum. Since $H$ is a complete $r$-pointed algebra in
$\cacat r$, there is a natural morphism $H \to k^r$, inducing an
algebra homomorphism
    \[ \pi: \mathcal O^A(\sfam) \to \osum_{1 \le i \le r}
    \enm_k(M_i) \]
By construction, there is a right action of $\mathcal O^A(\sfam)$ on
the family $\sfam$ extending the right action of $A$, in the sense that
the diagram
\begin{equation*}
\xymatrix{
A \ar[r]^-{\eta} \ar[rd]_-{\rho} & \mathcal O^A(\sfam) \ar[d]^-{\pi} \\
& \displaystyle \osum_{1 \le i \le r} \enm_k(M_i)
}
\end{equation*}
commutes. This makes it reasonable to call $\mathcal O^A(\sfam)$ the
algebra of observables.

\section{The generalized Burnside theorem}

Let $k$ be a field and let $A$ be a finite-dimensional associative
$k$-algebra. Then the simple right modules over $A$ are the simple
right modules over the semi-simple quotient algebra $A/\rad(A)$, where
$\rad(A)$ is the Jacobson radical of $A$. By the classification theory
for semi-simple algebras, it follows that there are finitely many
non-isomorphic simple right $A$-modules.

We consider the noncommutative deformation functor $\defs: \acat r \to
\sets$ of the family $\sfam = \{ M_1, M_2, \dots, M_r \}$ of simple
right $A$-modules. Clearly, $\sfam$ is a swarm, hence $\defs$ has a
formal moduli $(H, M_H)$, and we consider the commutative diagram
\begin{equation*}
\xymatrix{
A \ar[r]^-{\eta} \ar[rd]_-{\rho} & \mathcal O^A(\sfam) \ar[d]^-{\pi} \\
& \displaystyle \osum_{1 \le i \le r} \enm_k(M_i)
}
\end{equation*}
By a classical result, due to Burnside, the algebra homomorphism $\rho$
is surjective when $k$ is algebraically closed. This result is
conveniently stated in the following form:

\begin{thm}[Burnside's Theorem]
If $\enm_A(M_i) = k$ for $1 \le i \le r$, then $\rho$ is surjective. In
particular, $\rho$ is surjective when $k$ is algebraically closed.
\end{thm}
\begin{proof}
There is an obvious factorization $A \to A/\rad(A) \to \osum_i \,
\enm_k(M_i)$ of $\rho$. If $\enm_A(M_i)=k$ for $1 \le i \le r$, then
$A/\rad(A) \to \osum_i \, \enm_k(M_i)$ is an isomorphism by the
classification theory for semi-simple algebras. Since $\enm_A(M_i)$ is
a division ring of finite dimension over $k$, it is clear that
$\enm_A(M_i) = k$ whenever $k$ is algebraically closed.
\end{proof}

Let us write $\overline \rho: A/\rad A \to \osum_i \, \enm_k(M_i)$ for
the algebra homomorphism induced by $\rho$. We observe that $\rho$ is
surjective if and only if $\overline \rho$ is an isomorphism. Moreover,
let us write $J = \rad(\mathcal O^A(\sfam))$ for the Jacobson radical
of $\mathcal O^A(\sfam)$. Then we see that
    \[ J = ( \rad(H)_{ij} \widehat \otimes_k \hmm_k(M_i,M_j) ) =
    \ker(\pi) \]
Since $\rho(\rad A) = 0$ by definition, it follows that $\eta(\rad A)
\subseteq J$. Hence there are induced morphisms
    \[ \gr(\eta)_q: \rad(A)^q/\rad(A)^{q+1} \to J^q/J^{q+1} \]
for all $q \ge 0$. We may identify $\gr(\eta)_0$ with $\overline \rho$,
since $\mathcal O^A(\sfam)/J \cong \osum_i \, \enm_k(M_i)$. The
conclusion in Burnside's Theorem is therefore equivalent to the
statement that $\gr(\eta)_0$ is an isomorphism.

\begin{thm}[The Generalized Burnside Theorem]
Let $A$ be a finite-dimensional algebra over a field $k$, and let
$\sfam = \{ M_1, M_2, \dots, M_r \}$ be the family of simple right
$A$-modules. If $\enm_A(M_i) = k$ for $1 \le i \le r$, then $\eta: A
\to \mathcal O^A(\sfam)$ is an isomorphism. In particular, $\eta$ is an
isomorphism when $k$ is algebraically closed.
\end{thm}
\begin{proof}
It is enough to prove that $\eta$ is injective and that $\gr(\eta)_q$
is an isomorphism for $q = 0$ and $q = 1$, since $A$ and $\mathcal
O^A(\sfam)$ are complete in the $\rad(A)$-adic and $J$-adic topologies.
By Burnside's Theorem, we know that $\gr(\eta)_0$ is an isomorphism. To
prove that $\eta$ is injective, let us consider the kernel $\ker(\eta)
\subseteq A$. It is determined by the obstruction calculus of $\defs$;
see Laudal \cite{laud02}, Theorem 3.2 for details. When $A$ is
finite-dimensional, the right regular $A$-module $A_A$ has a
decomposition series
    \[ 0 = F_0 \subseteq F_1 \subseteq \dots \subseteq F_n = A_A \]
with $F_p/F_{p-1}$ a simple right $A$-module for $1 \le p \le n$. That
is, $A_A$ is an \emph{iterated extension} of the modules in $\sfam$.
This implies that $\eta$ is injective; see Laudal \cite{laud02},
Corollary 3.1. Finally, we must prove that $\gr(\eta)_1:
\rad(A)/\rad(A)^2 \to J/J^2$ is an isomorphism. This follows from the
Wedderburn-Malcev Theorem; see Laudal \cite{laud02}, Theorem 3.4 for
details.
\end{proof}

\section{Properties of the algebra of observables}

Let $A$ be a finite-dimensional algebra over a field $k$, and let
$\sfam = \{ M_1, \dots, M_r \}$ be any family of right $A$-modules of
finite dimension over $k$. Then $\sfam$ is a swarm, and we denote the
algebra of observables by $B = \mathcal O^A(\sfam)$. It is clear that
    \[ B/\rad(B) \cong \osum_i \, \enm_k(M_i) \]
is semi-simple, and it follows that $\sfam$ is the family of simple
right $B$-modules. In fact, one may show that $\sfam$ is a swarm of
$B$-modules, since $B$ is complete and $B/(\rad B)^n$ has finite
dimension over $k$ for all positive integers $n$.

\begin{prop}
If $k$ is an algebraically closed field, then the algebra homomorphism
$\eta_B: B \to \mathcal O^B(\sfam)$ is an isomorphism.
\end{prop}
\begin{proof}
Since $\sfam$ is a swarm of $A$-modules and of $B$-modules, we may
consider the commutative diagram
\begin{equation*}
\xymatrix{
A \ar[r]^-{\eta^A} \ar[rd]_-{\rho} & B = \mathcal O^A(\sfam)
\ar[d] \ar[r]^-{\eta^B} & C = \mathcal O^A(\sfam) \ar[dl] \\
& \displaystyle \osum_{1 \le i \le r} \enm_k(M_i) &
}
\end{equation*}
The algebra homomorphism $\eta^B$ induces maps $B/\rad(B)^n \to
C/\rad(C)^n$ for all $n \ge 1$. Since $k$ is algebraically closed and
$B/\rad(B)^n$ has finite dimension over $k$, it follows from the
Generalized Burnside Theorem that $B/\rad(B)^n \to C/\rad(C)^n$ is an
isomorphism for all $n \ge 1$. Hence $\eta^B$ is an isomorphism.
\end{proof}

In particular, the proposition implies that the assignment $(A, \sfam)
\mapsto (B, \sfam)$ is a closure operation when $k$ is algebraically
closed. In other words, the algebra $B = \mathcal O^A(\sfam)$ has the
following properties:
\begin{enumerate}
    \item The family $\sfam$ is the family the simple $B$-modules.
    \item The family $\sfam$ has exactly the same module-theoretic
        properties, in terms of (higher) extensions and Massey
        products, considered as a family of modules over $B$ as
        over $A$.
\end{enumerate}
Moreover, these properties characterizes the algebra $B = \mathcal
O^A(\sfam)$ of observables.

\section{Examples: Representations of ordered sets}

Let $k$ be an algebraically closed field, and let $\Lambda$ be a finite
ordered set. Then the algebra $A = k[\Lambda]$ is an associative
algebra of finite dimension over $k$. The category of right $A$-modules
is equivalent to the category of presheaves of vector spaces on
$\Lambda$, and the simple $A$-modules corresponds to the presheaves $\{
M_{\lambda}: \lambda \in \Lambda \}$ defined by $M_{\lambda}(\lambda) =
k$ and $M_{\lambda}(\lambda') = 0$ for $\lambda' \neq \lambda$. The
following results are well-known:
\begin{enumerate}
    \item If $\lambda > \lambda'$ in $\Lambda$ and $\{ \gamma \in
        \Lambda: \lambda > \gamma > \lambda' \} = \emptyset$, then
        $\ext^1_A(M_{\lambda}, M_{\lambda'}) \cong k$
    \item If $\{ \gamma \in \Lambda: \lambda \ge \gamma \ge
        \lambda' \}$ is a simple loop in $\Lambda$, then
        $\ext^2_A(M_{\lambda}, M_{\lambda'}) \cong k$
    \item In all other cases, $\ext^1_A(M_{\lambda}, M_{\lambda'})
        = \ext^2_A(M_{\lambda}, M_{\lambda'}) = 0$
\end{enumerate}

\subsection{A hereditary example}
Let us first consider the following ordered set. We label the elements
by natural numbers, and write $i \to j$ when $i > j$:
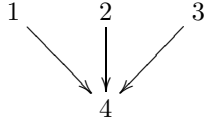
\begin{equation*}
\xymatrix{
1 \ar[dr] & 2 \ar[d] & 3 \ar[dl] \\
& 4 &
}
\end{equation*}
In this case, the simple modules are given by $\sfam = \{ M_1, M_2,
M_3, M_4 \}$, and we can easily compute the algebra $\mathcal
O^A(\sfam)$ of observables since $\ext^2_A(M_i,M_j) = 0$ for all $1 \le
i,j \le 4$. We obtain
    \[ \mathcal O^A(\sfam) = ( H_{ij} \widehat \otimes_k \hmm_k(M_i,M_j)
    ) \cong H \cong \left( \begin{matrix} k & 0 & 0 & k \\ 0 & k & 0 & k
    \\ 0 & 0 & k & k \\ 0 & 0 & 0 & k \end{matrix} \right) \]
It follows from the Generalized Burnside Theorem that $\eta: A \to
\mathcal O^A(\sfam)$ is an isomorphism. Hence we recover the algebra $A
\cong \mathcal O^A(\sfam) \cong H$.

\subsection{The diamond}
Let us also consider the following ordered set, called \emph{the
diamond}. We label the elements by natural numbers, and write $i \to j$
when $i > j$:
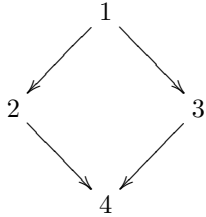
\begin{equation*}
\xymatrix{
& 1 \ar[dl] \ar[dr] & \\
2 \ar[dr] &  & 3 \ar[dl] \\
& 4 &
}
\end{equation*}
In this case, the simple modules are given by $\sfam = \{ M_1, M_2,
M_3, M_4 \}$. Since $\ext^2_A(M_1,M_4) \cong k$, we must compute the
cup-products
\begin{align*}
\ext^1_A(M_1,M_2) \cup \ext^1_A(M_2,M_4) & \to \ext^2_A(M_1,M_4) \\
\ext^1_A(M_1,M_3) \cup \ext^1_A(M_3,M_4) & \to \ext^2_A(M_1,M_4)
\end{align*}
in order to compute $H$. These cup-products are non-trivial; see Remark
3.2 in Laudal \cite{laud02} for details. Hence we obtain
    \[ \mathcal O^A(\sfam) = ( H_{ij} \widehat \otimes_k \hmm_k(M_i,M_j)
    ) \cong H \cong \left( \begin{matrix} k & k & k & k \\ 0 & k & 0 & k
    \\ 0 & 0 & k & k \\ 0 & 0 & 0 & k \end{matrix} \right) \]
Note that $H_{14}$ is two-dimensional at the tangent level and has a
relation. Also in this case, it follows from the Generalized Burnside
Theorem that $\eta: A \to \mathcal O^A(\sfam)$ is an isomorphism. Hence
we recover the algebra $A \cong \mathcal O^A(\sfam) \cong H$.

\bibliographystyle{plain}
\bibliography{defm-pz}

\end{document}